\documentclass[a4,12pt]{amsart}
\usepackage{amssymb,amstext,amsmath,amscd,amsthm,amsfonts,enumerate,graphicx,latexsym, comment}
\usepackage[usenames]{color}
\usepackage{tikz-cd}
\usepackage[all]{xy}
\usepackage{hyperref} 
\newcommand{\com}[1]{{\color{green} #1}}
\newcommand{\old}[1]{{\color{red} #1}}
 
\newtheorem{thm}{Theorem}[section]
\newtheorem{cor}[thm]{Corollary}
\newtheorem{lem}[thm]{Lemma}
\newtheorem{prop}[thm]{Proposition}
\theoremstyle{definition}
\newtheorem{dfn}[thm]{Definition}
\newtheorem{rem}[thm]{Remark}
\newtheorem{ques}[thm]{Question}
\newtheorem{prob}[thm]{Problem}

\newtheorem{conv}[thm]{Convention}
\newtheorem{ex}[thm]{Example}
\newtheorem{fact}[thm]{Fact}

\newtheorem{claim}{Claim}

\newtheorem*{claim*}{Claim}
\newtheorem{set}[thm]{Setup}
\theoremstyle{remark}

\numberwithin{equation}{thm}
\def\c{\mathfrak{c}}
\def\depth{\operatorname{depth}}

\def\Ext{\operatorname{Ext}}
\def\Hom{\operatorname{Hom}}
\def\I{\operatorname{\mathrm{I}}}
\def\T{\operatorname{\mathrm{T}}}
\def\m{\mathfrak{m}}
\def\n{\mathfrak{n}}

\def\tr{\operatorname{tr}}

\newcommand{\rme}{\mathrm{e}}

\newcommand{\fkc}{\mathfrak{c}}

\newcommand{\fkm}{\mathfrak{m}}
\newcommand{\fkn}{\mathfrak{n}}

\def\ol{\overline}
\def\rank{\operatorname{\mathrm{rank}}}
\def\Ref{\mathrm{Ref}}

\def\RT{\mathrm{RT}}

\def\Im{\operatorname{\mathrm{Im}}}
\def\Ker{\operatorname{\mathrm{Ker}}}

\begin{document}
\title[Set of trace ideals]{The set of trace ideals of curve singularities}
\author{Toshinori Kobayashi}
\address{Toshinori Kobayashi: School of Science and Technology, Meiji University, 1-1-1 Higashi-Mita, Tama-ku, Kawasaki-shi, Kanagawa 214-8571, Japan}
\email{toshinorikobayashi@icloud.com}

\author{Shinya Kumashiro}
\address{Shinya Kumashiro: National Institute of Technology, Oyama College
771 Nakakuki, Oyama, Tochigi, 323-0806, Japan}
\email{skumashiro@oyama-ct.ac.jp}

\thanks{2020 {\em Mathematics Subject Classification.} Primary: 13C05, Secondary: 13A15, 16S36}
\thanks{{\em Key words and phrases.} trace ideal, reflexive ideal, integrally closed ideal, curve singularity}

\thanks{Kobayashi was partly supported by JSPS Grant-in-Aid for JSPS Fellows 21J00567. Kumashiro was supported by JSPS KAKENHI Grant Number JP21K13766.}


\begin{abstract}
This paper mainly focuses on commutative local domains of dimension one. We then obtain a criterion for a ring to have a finite number of trace ideals in terms of integrally closed ideals. We also explore properties of such rings related to birational extensions, reflexive ideals, and reflexive Ulrich modules. Special attention is given in the case of numerical semigroup rings of non-gap four. We then obtain a criterion for a ring to have a finite number of reflexive ideals up to isomorphism. Non-domains arising from fiber products are also explored.
\end{abstract}

\maketitle

\section{Introduction}\label{section0}
Classification of ideals is one of the most classical problems in commutative ring theory.
It has been studied at least since the works
of Dedekind on rings of algebraic numbers.
If we consider a Dedekind domain, its ideal class group classifies the isomorphism classes of ideals.
If the considered ring is not integrally closed, the situation becomes more complicate.
For a one--dimensional local ring, classification of ideals relates to representations of the ring.
Actually, the result by Greuel and Kn\"orrer \cite{GK} shows that a one--dimensional Cohen--Macaulay local ring satisfying some mild assumptions has a finite number of isomorphism classes of ideals exactly when it is of finite representation type (see also \cite[Theorem 4.13]{lw}).
Here we say that a one--dimensional local ring is \textit{of finite representation type} if it has only finitely many torsion--free modules up to isomorphism (this definition is not the usual one, but equivalent to it under our assumption; see \cite{lw} for details).

In this paper, we study isomorphism classes of ideals in rings which are not necessary of finite representation type.
We then focus on a special class of ideals, namely, trace ideals.
Let us recall the definition of them.
Let $R$ be a commutative Noetherian local ring. The {\it trace ideal} of an $R$-module $M$ is defined to be the ideal
\begin{align*} 
\tr_R(M)=\sum_{f\in \Hom_R(M, R)} \Im f .
\end{align*}
Then an ideal $I$ in $R$ is called a {\it trace ideal} if $I=\mathrm{tr}_R(M)$ for some $R$-module $M$.
While the notion of trace ideals has long been used as a technical tool in commutative algebra, it itself has gained new attention in recent years \cite{dl, f, hr, kt, li}.
We should also mention recent use of trace ideals to develop theory of rings which are close to Gorenstein \cite{dkt, HHS, HKS2}.

One of the advantages to study trace ideals can be explained by a simple fact: if $I$ and $J$ are distinct trace ideals of a ring $R$, then they are non-isomorphic (see \cite[Corollary 1.2(a)]{hr} for example).
By this fact, to see how many non-isomorphic trace ideals there are, we only need to know what is the set of trace ideals.
We should mention a previous study \cite{GIK} on the set of trace ideals.
As a particular question, the following is raised naturally and explored in several papers:

\begin{ques} (\cite[Question 3.7]{f}, \cite[Question 7.16(1)]{dms},\cite{hr,k}) \label{q}
When does a Noetherian local ring have a finite number of trace ideals?
\end{ques}

In \cite{k}, the second author proved that if a local domain $R$ has a finite number of trace ideals, then $\dim R \le 1$ and $R$ is analytically unramified (\cite[Lemma 2.4 and Theorem 2.6]{k}). 
In the case of dimension one, it is also proven that analytically irreducible Arf local domains have a finite number of trace ideals (\cite[Corollary 5.5]{k}).
Here we refer to the paragraph before Corollary \ref{cor29} for the definition of Arf rings.
Note that the notion of Arf rings originates from a classification of certain singular points of plane curves (\cite{Lipman}).
We also remark that, under some suitable assumptions,
a Gorenstein local ring of dimension one has a finite number of trace ideals if and only if it is a ring of finite representation type \cite{dms,hr}.
However, other than Arf rings and rings of finite representation type, only few examples of rings having a finite number of trace ideals are known. 

Due to the previous results, we mainly deal with analytically irreducible local domains of dimension one.
Then our first aim is to give a complete answer to Question \ref{q} by assuming some mild conditions.
Let $(R,\m,k)$ be analytically irreducible local domains of dimension one.
Then the integral closure $\ol{R}$ of $R$ in the total ring of fraction $Q(R)$ of $R$ is finitely generated as an $R$-module and a local ring. Suppose that the canonical map $k \to \ol{R}/\n$, where $\n$ is the maximal ideal of $\ol{R}$, is an isomorphism (for instance, this fulfills if $k$ is algebraically closed). Let $\fkc=R:\ol{R}$ denote the {\it conductor} of $R$, where the colon is considered in $Q(R)$. Set $n=\ell_R(R/(R:\ol{R}))$, where $\ell_R(*)$ denotes the length. Let $v(x)$ denote the {\it value} of $x\in Q(R)$. 
For $0\le i\le n$, there exists a unique integrally closed ideal $I_i$ such that $\ell_R(R/I_i)=i$. 
Let 
\[ 
\T(R)=\{\text{nonzero trace ideals of a domain $R$}\}.
\]
With these notations and assumptions, we obtain a criterion for a ring to have finite number of trace ideals. 
%


\begin{thm} {\rm (Theorems \ref{th1} and \ref{t24})} \label{th02}
Let $n=\ell_R(R/\fkc)\ge 3$. 
If $k$ is infinite, then the following conditions are equivalent.
	\begin{enumerate}[\rm(1)]
		\item $\T(R)$ is a finite set. 
		\item All nonzero trace ideals are integrally closed ideals and contain the conductor $\fkc$, that is, $\T(R)=\{I_i \mid 0 \le i \le n\}$.
		\item For each $1\le i \le n-2$, there exists $q_i\in R$ such that $v(q_i)=\min\{v(x)\mid x\in I_i\}$ and $I_iI_{i+2}=q_iI_{i+2}$. 
	\end{enumerate}
If $R$ is a numerical semigroup ring $k[[H]]$ of a numerical semigroup 
\[
H=\{a_0=0<a_1<a_2<\cdots<a_n<a_{n+1}<a_{n+2}<\cdots\}\subseteq \mathbb{N},
\]
then the following is also equivalent to the above conditions.
\begin{enumerate}[\rm(4)] 
\item $a_j+a_{i+1}-a_i\in H$ for all $i\in\{1,\dots,n-2\}$ and $j\in\{i+2,\dots,n\}$.
\end{enumerate}
\end{thm}


As applications, we observe that there are abundant examples of rings having a finite number of trace ideals other than Arf rings (Examples \ref{e43} and \ref{ex36}). 
It is also observed that the finiteness of $\T(R)$ is inherited by that of $\T(I_i:I_i)$.
This can be regarded as an analogue of a characterization of Arf rings by Lipman (\cite[Theorem 2.2]{Lipman}).

By using Theorem \ref{th02}, we also try to understand the set $\Ref(R)$ of isomorphism classes of reflexive modules over a ring $R$.
Here, an $R$-module $M$ is called \textit{reflexive} if the canonical homomorphism $M \to \Hom_R(\Hom_R(M,R),R)$ is an isomorphism.
We remark that reflexive modules play an important role in representation theory of Cohen--Macaulay rings.
We refer to \cite{dms} for brief history of the study of reflexive modules.
In this context, it is natural to ask when $\Ref(R)$ is a finite set.
In this paper, we mainly restrict our attention to reflexive modules of rank one, that is, reflexive ideals.
Such a restriction is inspired by the following theorem given by Dao, Maitra, and Sridhar.
\begin{thm}{\rm (\cite[Propositions 7.3 and 7.9]{dms})} \label{th02_1}
Let $(R,\m,k)$ be a Cohen--Macaulay local ring of dimension one.
Assume $R$ is almost Gorenstein, contains $\mathbb{Q}$, and $k$ is algebraically closed.
Then the following conditions are equivalent.
\begin{enumerate}[\rm(1)]
\item $\Ref(R)$ is a finite set.
\item $R$ has a finite number of reflexive ideals up to isomorphism.
\item $\T(R)$ is a finite set.
\end{enumerate}
\end{thm}

We also note that if $R$ is Arf, then $\Ref(R)$ is a finite set (\cite[Theorem 3.5]{ik} and \cite[Corollary 3.5]{d}).

As a consequence of Theorem \ref{th02}, we deduce that under the same assumption as in Theorem \ref{th02}, $R$ has only finitely many reflexive ideals up to isomorphism (Theorem \ref{cor3.4}) provided that $\T(R)$ is a finite set.
In particular, we verify the implication (3)$\Rightarrow$(2) of Theorem \ref{th02_1} by assuming that $R$ is a domain instead of that $R$ is almost Gorenstein.


Special attention is given in the case of $n=4$. By observing Theorem \ref{th02}, we see that all rings $R$ have a finite number of trace ideals if $n\le 3$. On the other hand, it is also known that all rings $R$ have a finite number of reflexive ideals if $n\le 3$ (\cite[Theorem 6.8]{dms}). Hence, the case of $n=4$ is the next step to study the relation between trace ideals and reflexive ideals. In conclusion, we determine conditions under which numerical semigroup ring has finite reflexive ideals for $n=4$ as follows.

\begin{thm}{\rm (Theorem \ref{thm4.1})} \label{thm03}
	Let $R=k[[H]]$ be a numerical semigroup ring of a numerical semigroup 
	\[
	H=\{a_0=0<a_1<a_2<\cdots<a_n<a_{n+1}<a_{n+2}<\cdots\}\subseteq \mathbb{N},
	\] 
	where $k$ is a field. Suppose that $n=4$ and $k$ is infinite. Then the following conditions are equivalent.
	\begin{enumerate}[\rm(1)] 
		\item $R$ has a finite number of reflexive ideals up to isomorphism.
		\item $R$ has a finite number of reflexive trace ideals.
		\item All reflexive ideals are isomorphic to some monomial ideal containing the conductor $\fkc$.
		\item Either one of the following holds true:
		\begin{enumerate}[\rm(i)] 
			\item $a_2-a_1+a_3\ge a_4$, that is, $\T(R)$ is finite.
			\item $2a_3 - a_1 < a_4$.
		\end{enumerate}
	\end{enumerate}
\end{thm}

As a corollary, we obtain examples of a ring which has infinitely many trace ideals, but has a finite number of reflexive ideals (Example \ref{ex76}).
Note that such examples do not exist when the rings are assumed to be Arf or almost Gorenstein. 

Let us explain how we organize this paper. In Section \ref{section1}, we note several lemmas, which we use throughout this paper. In particular, we study an equality $IJ=qJ$ for ideals $I$, $J$ and $q\in I$. Recall that this equality is used to characterize the finiteness of trace ideals  in Theorem \ref{th02}(3).
In Section \ref{newsection2}, we prove Theorem \ref{th02}. In Section \ref{section2}, we apply Theorem \ref{th02} to numerical semigroup rings, and give examples.

The subject of Section \ref{section4} is a little different from that of other sections. According to our results, the case of analytically irreducible domains is well-explored. However, the case of non-domains is left open. Thus, in Section \ref{section4}, we examine the set of trace ideals of fiber products as a trial run. We describe the set of trace ideals containing a non-zerodivisor of fiber product $R_1\times_k R_2$ by those of $R_1$ and $R_2$ (Theorem \ref{th51}). 
Section \ref{section5} comes back to the main focus of this paper. We prove that for each ring $R$ having finite trace ideals, $R$ has a finite number of reflexive ideals up to isomorphism. 
We also investigate reflexive Ulrich modules under similar assumptions. In Section \ref{section6} we prove Theorem \ref{thm03}.

\begin{conv}
In the rest of this paper, all rings are commutative Noetherian rings with identity.
Let $R$ be a ring. Then, $Q(R)$ and $\ol{R}$ denote the total ring of fraction of $R$ and the integral closure of $R$, respectively. We denote by $R^\times$ the set of units of $R$.

We say that $I$ is a {\it fractional ideal} if $I$ is a finitely generated $R$-submodule of $Q(R)$ containing a non-zerodivisor of $R$. For fractional ideals $I$ and $J$, we denote by $I:J$ the fractional ideal $\{x\in Q(R) \mid xJ\subseteq I\}$. It is known that an isomorphism $I:J\cong \Hom_R(J, I)$ is given by the correspondence $x \mapsto \hat{x}$, where $\hat{x}$ denotes the multiplication map of $x\in I:J$ (see \cite[Lemma 2.1]{HK}). We say that an ideal $I$ is {\it regular} if $I$ contains a non-zerodivisor of $R$. 
For a finitely generated $R$-module $M$, $\ell_R(M)$ (resp. $\mu_R(M)$, $\rme(M)$) denotes the length of $M$ (resp. the number of minimal generators of $M$, the multiplicity of $M$).
Set
\[ 
\T(R)=\{\text{regular trace ideals of $R$}\}.
\]
Note that $\T(R)$ is precisely the set of nonzero trace ideals if $R$ is a domain.
In addition, if $\ol{R}$ is finitely generated as an $R$-module, then $\fkc=R:\ol{R}$ denotes the {\it conductor} of $R$. 
\end{conv}

\section{Preliminaries}\label{section1}

Let $(R, \fkm)$ be a Cohen-Macaulay local ring of dimension one. 
The aim of this section is to prepare several lemmas, which are used from the next section onward.

\begin{lem}{\rm (\cite[Corollary 2.2]{GIK})}\label{prelim1}
Let $I$ be a regular ideal of $R$. The following are equivalent.
\begin{enumerate}[\rm(1)] 
\item $I$ is a trace ideal.
\item $(R:I)I=I$.
\item $R:I=I:I$.
\end{enumerate}
\end{lem}

\begin{lem}\label{prelim2}
Let $I$ and $J$ be regular trace ideals of $R$ such that $I\subseteq J$. Then, $J:J\subseteq I:I$. In particular, $(J:J)I=I$.
\end{lem}

\begin{proof}
Since $I\subseteq J$, we have $J:J=R:J\subseteq R:I=I:I$ by Lemma \ref{prelim1}. By noting that $R\subseteq J:J$, we have $I\subseteq (J:J)I \subseteq (I:I)I=I$.
\end{proof}

Next we consider an equality $IJ=qJ$, where $I$ and $J$ are regular ideals of $R$ and $q\in I$. Such an equality plays a key role in our characterization of the finiteness of the set of regular trace ideals given in the next section (see Theorem \ref{th1}).

\begin{lem}\label{lem2.9}
Let $I$ and $J$ be regular ideals of $R$ and $q\in I$ is a non-zerodivisor of $R$. Then $J:I\subseteq q^{-1} J$.
Furthermore, $J:I= q^{-1} J$ if and only if $IJ=qJ$. 
In particular, $I:I\subseteq q^{-1} I$, and $I:I= q^{-1} I$ if and only if $I^2=qI$. 
\end{lem}

\begin{proof}
Let $x\in J:I$. Then $qx\in Ix\subseteq J$. It follows that $x\in q^{-1} J$.
Furthermore, 
\begin{align*}
q^{-1} J=J:I \iff  q^{-1} J \subseteq J:I\iff q^{-1} IJ \subseteq I \iff IJ\subseteq qI \iff IJ=qI.
\end{align*}
\end{proof}

\begin{lem}\label{prelim3}
Let $I$ and $J$ be regular ideals of $R$.
Suppose that there exists an element $q\in I$ of $R$ such that $IJ=qJ$.
Then for each regular trace ideal $L$ with $L\subseteq \tr_R(J)$, an equality $IL=qL$ holds.
\end{lem}

\begin{proof}
Note that $q$ is a non-zerodivisor of $R$ since $IJ$ is a regular ideal and $IJ=qJ\subseteq (q)$.
Consider the evaluation map $\mathrm{ev}\colon(R:J)\otimes_R J\to \tr_R(J)$; $x\otimes y \mapsto xy$, where $x\in R:J$ and $y\in J$.
It induces a surjection $J^{\oplus n} \to \tr_R(J)$ for some $n$.
Tensoring $R/(q)$, we have a surjection $(J/qJ)^{\oplus n} \to \tr_R(J)/q\tr_R(J)$.
Since $IJ=qJ$, $J/qJ$ is annihilated by $I$. Hence, $I\tr_R(J)/q \tr_R(J)=0$, that is, $I\tr_R(J)=q \tr_R(J)$. 

Let $L$ be a regular trace ideal $L$ with $L\subseteq \tr_R(J)$. Note that $I\tr_R(J)=q \tr_R(J)$ implies $q^{-1}I\subseteq \tr_R(J):\tr_R(J)$. Then, we obtain that $q^{-1}I\subseteq \tr_R(J):\tr_R(J)\subseteq L:L$ by Lemma \ref{prelim2}. It follows that $q^{-1}IL\subseteq L$; hence, $IL=qL$.
\end{proof}

Next we give a correspondence between certain subsets of $\T(R)$ and $\T(I:I)$ for a pair of ideals $I$ and $J$ with $IJ=qJ$.

\begin{prop} \label{pp1}
Let $I$ be a regular trace ideal of $R$, and let $J$ be a regular ideal of $R$ with $J\subseteq I$.
Suppose that there exists an element $q\in I$ such that $IJ=qJ$. Then 
\[
\{X\in\T(I:I)\mid X\subseteq J:I\}=\{q^{-1}Y\mid Y\in \T(R)\text{ such that }Y\subseteq J\}.
\]
\end{prop}

\begin{proof}
Set $S:=I:I$. Note that $q^{-1}J= J:I\subseteq S$ by Lemma \ref{lem2.9}.

$(\supseteq)$: Let $Y\in \T(R)$ such that $Y\subseteq J$. Then $q^{-1}Y \subseteq q^{-1}J \subseteq S$. On the other hand, since $Y\subseteq I$, $q^{-1}YS=q^{-1}Y$ by Lemma \ref{prelim2}. Hence, $q^{-1}Y$ is an ideal of $S$.
Check the equalities
\[
(S:q^{-1}Y)q^{-1}Y=[(R:I):q^{-1}Y]q^{-1}Y=(R:q^{-1}IY)q^{-1}Y=(R:Y)q^{-1}Y=q^{-1}Y,
\]
where the third equality follows from $IY=qY$ by Lemma \ref{prelim3}.
It follows that $q^{-1}Y\in \T(S)$.
By Lemma \ref{prelim3} again, $I(q^{-1}Y)=Y\subseteq J$. Thus, $q^{-1}Y\subseteq J:I$.

$(\subseteq)$: Let $X\in \T(S)$ such that $X\subseteq J:I$.
Then 
\[
IX=I(S:X)X=I[(R:I):X]X=(R:IX)IX.
\]
This means that $IX\in \T(R)$.
Since $X\subseteq J:I$, $q^{-1}IX\subseteq q^{-1}J \subseteq S$.
Hence, $q^{-1}I\subseteq S:X=X:X$.
It follows that $q^{-1}IX\subseteq X$.
Therefore, we obtain that $X=q^{-1}IX$.
\end{proof}

We have some applications of Proposition \ref{pp1}.
First we deal with the case of $I=J$.
Then we obtain the following description of $\T(I:I)$.

\begin{cor} \label{c111}
Let $I$ be a regular trace ideal of $R$.
Assume that there exists an element $q\in I$ such that $I^2=qI$.
Then
\[
\T(I:I)=\{q^{-1}Y\mid Y\in \T(R)\text{ such that }Y\subseteq I\}.
\]
\end{cor}

\begin{proof}
We may apply Proposition \ref{pp1} by letting $J=I$.
\end{proof}



Next we consider the case of $\ell_R(I/J)=2$.
If $I:I$ is local, then we get a description of $\T(I:I)$ similar to Corollary \ref{c111}.
Before stating it, we prepare a lemma.

\begin{lem}\label{lem2.10}
Let $I,J$ be regular trace ideals of $R$ such that $J\subseteq I$ and $\ell_R(I/J)=2$.
Assume that $I:I$ is a local ring, and there exists an element $q\in I\setminus J$ such that $IJ=qJ$ and $I^2\not=qI$.
 Then the maximal ideal of $I:I$ is $q^{-1}J$.
\end{lem}

\begin{proof}
Set $S:=I:I$.
By our assumption, we have $q^{-1}IJ=J \subseteq I$.
It follows by Lemma \ref{lem2.9} that 
\[
q^{-1}J \subsetneq S \subsetneq q^{-1}I.
\]
Here the first inequality follows by observing $1\in S \setminus q^{-1}J$.
By noting that $\ell_R(q^{-1}I/q^{-1}J)=\ell_R(I/J)=2$, we obtain that $\ell_R(S/q^{-1}J) = 1$.
On the other hand, $q^{-1}J$ is an ideal of $S$ by Lemma \ref{prelim2}.
Hence, $0<\ell_S(S/q^{-1}J)\le \ell_R(S/q^{-1}J)=1$.
This shows that $q^{-1}J$ is the maximal ideal of $S$.
\end{proof}

\begin{cor} \label{c112}
Let $I,J$ be nonzero trace ideals of $R$ such that $J\subseteq I$ and $\ell_R(I/J)=2$.
Assume that $I:I$ is a local ring, and there exists an element $q\in I\setminus J$ such that $IJ=qJ$ and $I^2\not=qI$.
Then
\[
\T(I:I)=\{q^{-1}Y\mid Y\in \T(R) \text{ such that }Y\subseteq J\}\cup\{I:I\}.
\]
\end{cor}

\begin{proof}
Since $I:I$ is a local ring with the maximal ideal $q^{-1}J$, it follows that $\T(I:I)=\{X\in \T(I:I)\mid X\subseteq q^{-1}J\}\cup\{I:I\}$.
Note that $q^{-1}J\subseteq J:I\subsetneq I:I$.
So, applying Proposition \ref{pp1}, we see the equality $\{X\in \T(I:I)\mid X\subseteq q^{-1}J\}=\{q^{-1}Y\mid Y\in \T(R) \text{ such that }Y\subseteq J\}$.
\end{proof}

\section{Trace ideals of curve singularities}\label{newsection2}

In this section, let $(R,\m,k)$ be an analytically irreducible local domain of dimension one, that is, $\ol{R}$ is finitely generated as an $R$-module and $\ol{R}$ is a local ring (hence, $\ol{R}$ is a discrete valuation ring). 
We assume that the canonical map $k \to \ol{R}/\n$, where $\n$ is the maximal ideal of $\ol{R}$, is an isomorphism (e.g. $k$ is an algebraically closed field or $R$ is a numerical semigroup ring). With this assumption, we investigate the structure of $\T(R)$.
We use the following notations:

\begin{set} \label{setup}
\begin{enumerate}
\item $v\colon Q(R)\to\mathbb{Z}\cup\{\infty\}$ denotes the {\it normalized valuation} associated to $\ol{R}$. 
\item $v(R)=\{v(r) \mid 0\ne r\in R\}$ denotes the \textit{value semigroup} of $R$. Set $H=v(R)$.
\item We write $H=\{a_0=0<a_1<a_2<\cdots<a_n<a_{n+1}<a_{n+2}<\cdots\}$. Note that there exists an integer $n$ such that $a_{n+i}=a_{n} + i$ for all $i\ge 0$. We choose such $n$ as small as possible.
\end{enumerate}
\end{set}

In addition, let 
\begin{align*}
\T(R)&=\{\text{regular trace ideals of $R$}\} \text{ as previous section, and }\\
\I(R)&=\{\text{integrally closed ideals of $R$ containing $\fkc$}\}.
\end{align*}
By letting $I_i:=\{r\in R\mid v(r)\ge a_i\}$ for all $i\in\{0,1,\dots,n\}$, we obtain that 
\[
\I(R)=\{I_i\mid i=0,\dots,n\}.
\]
Note that $I_0=R$, $I_1=\m$, and $I_n=\c$. 

\begin{rem}\label{remrem2.2}
\begin{enumerate}[\rm(1)] 
\item Let $r\in R$ be an element such that $v(r)=a_i$, where $0 \le i \le n$. Then the equality $I_i=(r)+I_{i+1}$ holds. In particular, $\ell_R(I_i/I_{i+1})=1$.
\item The integer $n$ appearing in Setup \ref{setup} (3) is equal to $\ell_R(R/\c)$.
Indeed, we have equalities $n=\ell_R(R/I_1)+\ell_R(I_1/I_2)+\cdots+\ell_R(I_{n-1}/I_n)=\ell_R(R/\c)$.
\end{enumerate}
\end{rem}


\begin{fact}\label{fact2.3zz}
\begin{enumerate}[\rm(1)] 
\item {\rm (\cite[Proposition 2.2]{hr})}: Let $I$ be a regular trace ideal of $R$. Then,  $I$ contains $\c$.
\item {\rm (\cite[Theorem 1]{c})}: $\I(R)$ is a subset of $\T(R)$.
\item {\rm (\cite[Theorem 6.8]{dms})}: If $n=2$, then $\T(R)=\{R,\m,\c\}=\I(R)$.
\end{enumerate}
\end{fact}

On the basis of the above facts, we aim to explore the finiteness of $\T(R)$. Let us start the following technical proposition.

\begin{prop} \label{pp}
Let $1\le i \le n-2$.
The following conditions are equivalent.
\begin{enumerate}[\rm(1)]
\item For any element $r\in R$ with $v(r)=a_i$, the equality $I_iI_{i+2}=rI_{i+2}$ holds.
\item There exists an element $q\in R$ such that $v(q)=a_i$ and the equality $I_iI_{i+2}=qI_{i+2}$ holds. 
\end{enumerate}
Assume $i\le n-3$ and $s\in R$ is an element such that $v(s)=a_{i+1}$ and $I_{i+1}I_{i+3}=sI_{i+3}$.
Then the following is also equivalent to both of the above conditions.
\begin{enumerate}[\rm(3)]
\item There exists an element $q\in R$ such that $v(q)=a_i$ and the inclusion $sI_{i+2}\subseteq (q)$ holds. 
\end{enumerate}
\end{prop}

\begin{proof}
(1)$\Rightarrow$(2): This is obvious.

(2)$\Rightarrow$(1): Let $r\in R$ with $v(r)=a_i$. We first prove the following claim.

\begin{claim}\label{claim1insec2}
$I_{i+1}I_j\subseteq (r)$ for all $i+2 \le j\le n$.
\end{claim}
\begin{proof}[Proof of Claim \ref{claim1insec2}]
We prove Claim 1 by descending induction on $j$.
If $j=n$, then $r^{-1}I_{i+1}I_n \subseteq \{x\in Q(R)\mid v(x)\ge a_n\}=\c$; hence, $I_{i+1}I_n\subseteq r\c\subseteq (r)$.
Suppose that $j<n$ and $I_{i+1}I_{j+1}\subseteq (r)$. 
By noting that $q^{-1}I_{i+1}I_j\subseteq q^{-1}I_{i}I_{i+2} = I_{i+2}\subseteq R$, where the equality follows from the assumption (2), we obtain that
\[
q^{-1}I_{i+1}I_j\subseteq R\cap \{x\in Q(R)\mid v(x)\ge a_j+(a_{i+1}-a_i)\}\subseteq I_{j+1}.
\]
This means $I_{i+1}I_j\subseteq qI_{j+1}$.
Choose a unit $u\in R^\times$ such that $q-ur\in I_{i+1}$.
Then $I_{i+1}I_j\subseteq qI_{j+1}\subseteq urI_{j+1}+I_{i+1}I_{j+1}$.
By the induction hypothesis, $I_{i+1}I_{j+1}\subseteq (r)$.
Hence we get $I_{i+1}I_j\subseteq (r)$.
Thus we may proceed the induction.
\end{proof}

By Claim \ref{claim1insec2}, we obtain that $I_{i+1}I_{i+2}\subseteq (r)$.
Remembering that $I_i=(r)+I_{i+1}$,
it follows that $r^{-1}I_iI_{i+2}\subseteq R\cap \{x\in Q(R)\mid v(x) \ge a_{i+2}\}=I_{i+2}$, that is, $I_iI_{i+2}=rI_{i+2}$.

Now assume $i\le n-3$ and $s\in R$ is an element such that $v(s)=a_{i+1}$ and $I_{i+1}I_{i+3}=sI_{i+3}$.
The implication (2)$\Rightarrow$(3) is clear.
We consider the converse direction (3)$\Rightarrow$(2).
Our assumption (3) says $q^{-1}sI_{i+2}\subseteq R\cap \{x\in Q(R)\mid v(x)\ge a_{i+3}\}=I_{i+3}$.
On the other hand, we see inclusions
$
q^{-1}I_{i+2}I_{i+2}=s^{-1}I_{i+2}(q^{-1}sI_{i+2})\subseteq s^{-1}I_{i+2}I_{i+3}\subseteq I_{i+3}
$.
Here the last inclusion follows by the assumption on $s$. 
Hence, we have inclusions $I_{i+2}I_{i+2}, sI_{i+2}\subseteq qI_{i+3}$.
Remembering $I_i=(q)+(s)+I_{i+2}$, we get $I_iI_{i+2}=qI_{i+2}+sI_{i+2}+I_{i+2}I_{i+2}=qI_{i+2}$.
\end{proof}

\begin{lem}\label{lem2.5}
When $i=n-1$ or $i=n$, the equality $I_{i}^2=q_{i}I_{i}$ holds. 
\end{lem}

\begin{proof}
This is clear if $i=n$. 
Assume that $i=n-1$. Since $I_{n-1}=(q_{n-1})+I_n$, we only need to check $I_n^2\subseteq q_{n-1}I_{n-1}$. 
Since $q_{n-1}^{-1}I_n^2\subseteq \{x\in Q(R)\mid v(x)\ge 2a_n-a_{n-1}(\ge a_n)\}\subseteq I_n\subseteq I_{n-1}$, we see the inclusion $I_n^2\subseteq q_{n-1}I_{n-1}$.
\end{proof}

Let $n\ge 4$, and fix an integer $i$ such that $1\le i\le n-3$.
Take elements $q,q'\in R$ such that $v(q)=a_i$ and $v(q')=a_{i+1}$.
For each $\alpha\in R$, we set 
\[
J_\alpha^{(i)}:=(q+\alpha q')+I_{i+2}.
\]
Although the ideal above depends on the choice of $q$ and $q'$ (not only on $i$ and $\alpha$), we use this notation to avoid complications.
The following proposition shows that $J_\alpha^{(i)}$ are trace ideals.


\begin{prop} \label{ppp}
Assume that $I_iI_{i+2}\not=qI_{i+2}$ and $I_{i+1}I_{i+3}=q'I_{i+3}$.
Then the following hold true.
\begin{enumerate}[\rm(1)]
\item For each $\alpha \in R$, $J_\alpha^{(i)}\in \T(R)$.
\item If $\alpha-\beta\not\in \fkm$, then $J_\alpha^{(i)}\not\supseteq J_\beta^{(i)}$.
\item If $k$ is infinite, then $\T(R)$ is an infinite set.
\end{enumerate}
\end{prop}

\begin{proof}
Set $f:=q+\alpha q'$. 

(1): Let $g\in R:J_\alpha^{(i)}$. It is enough to  prove that $g\in J_\alpha^{(i)}:J_\alpha^{(i)}$ (see Lemma \ref{prelim1}).
Since $\c=I_n\subseteq J_\alpha^{(i)}$, $g\in R:\c=\overline{R}$.
Hence, we can write $g=u+h$, where $h\in \overline{R}$ with $v(h)\ge 1$ and either $u=0$ or $u\in R^\times$. Indeed, if $v(g)>0$, then we can choose $u$ as $0$ and $h$ as $g$. If $v(g)=0$, then there exists $u\in R$ with $v(u)=0$ such that $v(g-u)>0$. Thus, we can define $h$ as $g-u$.
By noting that $u\in R$ and $g\in R:J_\alpha^{(i)}$, we have $h\in R:J_\alpha^{(i)}$. 
Moreover, to show $g\in J_\alpha^{(i)}:J_\alpha^{(i)}$, it is enough to check $h\in J_\alpha^{(i)}:J_\alpha^{(i)}$.
We may assume $h\not=0$.

Observe that $h\in R:J_\alpha^{(i)}\subseteq R:I_{i+2}$, and so $hI_{i+2}\subseteq (R:I_{i+2})I_{i+2}=I_{i+2}$.
Since $v(h)\ge 1$, we can obtain a more strict inclusion $hI_{i+2}\subseteq I_{i+3}$.
As $f\in J_\alpha^{(i)}$, we have $hf\in R$.
Thus $v(h)+a_i=v(fh)\in H$.
Since $v(h)\ge 1$, this implies that either $v(h)+a_i=a_{i+1}$ or $v(h)+a_i\ge a_{i+2}$.
Suppose $v(h)+a_i=a_{i+1}$.
Then $v(fh)=v(f)+v(h)=a_{i+1}$; hence, $I_{i+1}I_{i+3}=fhI_{i+3}$ by Proposition \ref{pp}.
On the other hand, $fhI_{i+2}=f(hI_{i+2})\subseteq fI_{i+3}$.
By Proposition \ref{pp} (3)$\Rightarrow$(1), we reach an equality $I_iI_{i+2}=qI_{i+2}$. This contradicts our assumption. It follows that $v(h)+a_i\ge a_{i+2}$. 

Hence, $hf\in R \cap\{x\in Q(R)\mid v(x)\ge a_{i+2}\}=I_{i+2}\subseteq J_\alpha^{(i)}$.
By combining this inclusion with the inclusion $hI_{i+2}\subseteq I_{i+3}\subseteq J_\alpha^{(i)}$, we obtain the desired inclusion $hJ_\alpha^{(i)}\subseteq J_\alpha^{(i)}$.
Therefore, we conclude that $J_\alpha^{(i)}$ is a trace ideal of $R$.

(2): 
Suppose that $\alpha-\beta\not\in \fkm$ and $J_\alpha^{(i)}\supseteq J_\beta^{(i)}$.
Then $(\alpha-\beta)q'=(q+\alpha q') - (q+\beta q')\in J_\alpha^{(i)}$. 
By noting that $\alpha-\beta$ is a unit of $R$, $q'\in J_\alpha^{(i)}$.
This means that there exists $x\in R$ and $y\in I_{i+2}$ such that $q'=xf + y$ (note that $f=q+\alpha q'$). 
Since $v(q')=a_{i+1}$, we have $v(q'-y)=a_{i+1}$. It follows that $I_{i+1}I_{i+3}=(q'-y)I_{i+3}$ by Proposition \ref{pp}. On the other hand, since $q'-y=xf$, we can also observe that $(q'-y)I_{i+2}=xfI_{i+2}\subseteq (f)$.
Thus, by Proposition \ref{pp} (3)$\Rightarrow$(1), we get $I_iI_{i+2}=qI_{i+2}$. This contradicts our assumption.

(3): By (1) and (2), any pair of nonzero distinct representatives $\alpha$ and $\beta$ of the residue field $k=R/\fkm$ provides distinct trace ideals $J_\alpha^{(i)}$ and $J_\beta^{(i)}$. Hence, there are trace ideals more than the cardinality of $k$.
\end{proof}

\begin{cor} \label{cc}
Let $n\ge 4$.
Assume there exists $1 \le i \le n-3$ such that $I_iI_{i+2}\not=qI_{i+2}$ for some (any) $q\in R$ with $v(q)=a_i$.
Then the following hold true.
\begin{enumerate}[\rm(1)]
\item $\I(R)\subsetneq \T(R)$.
\item If $k$ is infinite, then $\T(R)$ is an infinite set.
\end{enumerate}
\end{cor}

\begin{proof}
We first note that for any element $p\in I_{n-2}$ with $v(q)=n-2$, the equality $I_{n-2}I_{n}=qI_n$ always holds.
Thus we may pick $i$ as the biggest integer such that the inequality $I_iI_{i+2}\not=qI_{i+2}$ holds for any $q\in R$ with $v(q)=a_i$.
In particular, for such $i$, we have $I_{i+1}I_{i+3}=q'I_{i+3}$ for any $q'\in R$ with $v(q')=a_{i+1}$.

(1): By Proposition \ref{ppp}(1), we have $J_1^{(i)}=(q+q')+I_{i+2}\in \T(R)$. On the other hand, by Proposition \ref{ppp}(2), $J_1^{(i)}$ cannot contain $J_0^{(i)}$. In particular, $J_1^{(i)}\ne I_{i}$ since $J_0^{(i)}\subseteq I_{i}$. This shows that $J_1^{(1)}\not\in \I(R)$.
Hence we obtain that $\I(R)\not=\T(R)$. By recalling Fact \ref{fact2.3zz}, this proves $\I(R)\subsetneq \T(R)$.

(2): Now we assume $k$ is infinite.
Then, by Proposition \ref{ppp}(3), $\T(R)$ contains an infinite subset $\{J_\alpha^{(i)} \mid \text{$\alpha$ is a nonzero representative of $k$}\}$.
\end{proof}

Here we achieve the main theorem of this section.

\begin{thm} \label{th1}
Let $n\ge 3$. 
The following conditions are equivalent.
\begin{enumerate}[\rm(1)]
\item $\T(R)=\I(R)$.
\item For each $1\le i \le n-2$, there exists an element $q_i$ such that $v(q_i)=a_i$ and $I_iI_{i+2}=q_iI_{i+2}$.
\item For each $1\le i \le n-2$ and each element $q_i\in R$ with $v(q_i)=a_i$, the equality $I_iI_{i+2}=q_iI_{i+2}$ holds.
\end{enumerate}
If the residue field $k$ is infinite, then the following is also equivalent to the above conditions.
\begin{enumerate}[\rm(4)] 
\item $\T(R)$ is a finite set. 
\end{enumerate}
\end{thm}

\begin{proof}
(1)$\Rightarrow$(4): This is trivial. 

(4)$\Rightarrow$(2): Assume that $k$ is infinite. If $n=3$, then the condition (2) is automatically satisfied by Lemma \ref{lem2.5}. Hence, the assertion holds true. If $n\ge 4$, then the assertion holds by Corollary \ref{cc}. 

Hence, it is enough to prove that (1), (2), (3) are equivalent. (2)$\Leftrightarrow$(3) follows from Proposition \ref{pp}. Note that for the case of $n=3$, the condition (2) is automatically satisfied by Lemma \ref{lem2.5}. Hence, it is enough to prove the implications (1)$\Rightarrow$(2) for $n\ge 4$ and (3) $\Rightarrow$ (1) for $n\ge 3$.

(1)$\Rightarrow$(2): This implication follows from Corollary \ref{cc}.

(3)$\Rightarrow$(1): Let $I\in \T(R)$.
We aim to prove $I\in \I(R)$.
Let $i$ be an integer such that $a_i=\min\{v(x)\mid x\in I\}$, and we choose $q_i\in I$ such that $v(q_i)=a_i$. Note that $I$ contains $\fkc$ by Fact \ref{fact2.3zz}(1). Hence, $I=\fkc$ if $i=n$. If $i=n-1$, we obtain that $\fkc \subsetneq I \subseteq I_{n-1}$. It follows that $I=I_{n-1}$. Hence, we may assume that $1 \le i \le n-2$.

Since $I$ contains $\c$, we can write $I=(q_i,f_2,\dots,f_l)+\c$ for some $l\ge 2$, where  $f_2,\dots,f_l\in I_{i+1}$. 

\begin{claim}\label{claim2insec2}
$I$ contains $I_j$ for each $j\in\{i+2,\dots,n\}$.
\end{claim}

\begin{proof}[Proof of Claim \ref{claim2insec2}]
We proceed by descending induction on $j$.
The case of $j=n$ is trivial.
Suppose that $j<n$ and $I\supseteq I_{j+1}$. 
Since $j\ge i+2$ and $i+1\ge i$, we have $I_{i+1}I_j\subseteq I_iI_{i+2}=q_iI_{i+2}$.
In other words, $q_i^{-1}I_jI_{i+1}\subseteq I_{i+2}\subseteq R$. 
Hence, by noting that $f_2,\dots,f_l\in I_{i+1}$, we obtain that $q_i^{-1}I_j [(f_2,\dots,f_l)+\c]\subseteq q_i^{-1}I_j I_{i+1}\subseteq R$.
It follows that 
\[
q_i^{-1}I_j I= q_i^{-1}I_j [(q_i, f_2,\dots,f_l)+\c]\subseteq I_j + q_i^{-1}I_j I_{i+1}\subseteq R.
\] 
In other words, we have $q_i^{-1}I_j\subseteq R:I=I:I$, where the last equality follows from Lemma \ref{prelim1}. Therefore, we obtain that $I_j=q_i^{-1}I_j q_i \subseteq q_i^{-1}I_j I \subseteq (I:I)I=I$.
\end{proof}

By Claim \ref{claim2insec2}, we have $I_{i+2} \subsetneq I \subseteq I_i$. 
By noting that $\ell_R(I_i/I_{i+2})=2$ (see Remark \ref{remrem2.2}) and there is nothing to prove if $I=I_i$, we may write $I=(q_i)+I_{i+2}$. Let $q_{i+1}\in R$ such that $v(q_{i+1})=a_{i+1}$. 
By noting that $I_iI_{i+2}=q_iI_{i+2}$, we obtain that 
\begin{center}
$(q_i^{-1}q_{i+1})I_{i+2}\subseteq (q_i^{-1}) I_{i+1} I_{i+2}\subseteq (q_i^{-1}) I_i I_{i+2} = (q_i^{-1}) q_i I_{i+2}=I_{i+2}\subseteq R$ \ \ and \ \ $(q_i^{-1}q_{i+1})q_i=q_{i+1}\in R$.
\end{center}
From the above inclusions, we deduce $q_i^{-1}q_{i+1}\in R:I$. Hence, $q_{i+1}\in (R:I)I=I$.
Thus, we conclude $I=(q_i,q_{i+1})+I_{i+2}=I_i$.
\end{proof}

\begin{cor}\label{cor23}
Let $n\le 3$.
Then the equality $\T(R)=\I(R)$ holds.
In particular, $\T(R)$ is a finite set.
\end{cor}

We aim to apply Theorem \ref{th1} to Arf rings.
Here we say that a local ring $(R,\m)$ is \textit{Arf} if every regular integrally closed ideal $I$ satisfies $I^2=xI$ for some $x\in I$ (cf. \cite[Theorem 2.2]{Lipman}).

\begin{cor}\label{cor29} {\rm (\cite[Proposition 3.1]{ik})}
If $R$ is an Arf ring, then $\T(R)=\I(R)$.
\end{cor}

\begin{proof}
Since $R$ is an Arf ring, $I_i^2=q_i I_i$ for all $1 \le i \le n-2$. By Lemma \ref{prelim3}, $I_iI_{i+2}=q_iI_{i+2}$ for all $1 \le i \le n-2$. Hence, the assertion follows from Theorem \ref{th1}.
\end{proof}

The following theorem shows that the finiteness of $\T(R)$ is inherited by that of $\T(I_i:~I_i)$. 

\begin{thm} \label{tt}
Assume the equality $\I(R)=\T(R)$ holds.
Let $1 \le i \le n$ and $q_i\in R$ be an element such that $v(q_i)=a_i$.
Then
\[
\T(I_i:I_i)=
\begin{cases}
\{ q_i^{-1}I_{j} \mid i \le j \le n \} & \text{if $I_i^2=q_iI_i$} \\
\{ q_i^{-1}I_{j} \mid i+2 \le j \le n \} \cup \{ I_i:I_i\} & \text{if $I_i^2\ne q_iI_i$.}
\end{cases}
\]
In particular, $\T(I_i:I_i)$ is a finite set.
\end{thm}

\begin{proof}
Note that all intermediate rings between $R$ and $\ol{R}$ is a local ring because $\ol{R}$ is a local ring and finitely generated as an $R$-module. In particular, $I_i:I_i$ is a local ring for all $1 \le i \le n$.

Suppose that either $i=n$ or $n-1$. Then, the equality $I_i^2=q_iI_i$ holds by Lemma \ref{lem2.5}. 

Thus, the assertion follows by Corollary \ref{c111}.
Now let $1\le i\le n-2$.
By Theorem \ref{th1}, we have an equality $I_iI_{i+2}=q_iI_{i+2}$.
Note that $q_i\in I_i\setminus I_{i+2}$ and $\ell(I_i/I_{i+2})=2$.
Therefore, the assertion can be derived from Corollary \ref{c112}.
\end{proof}

Note that the converse of Theorem \ref{tt} does not hold in general:

\begin{ex}
Let $R=k[[t^5,t^6,t^7]]$ be a numerical semigroup ring over an infinite field $k$.
Then $\T(R)$ is infinite (see Example \ref{ex76}), but as $\m:\m$ is equal to $k[[t^5,t^6,t^7,t^8,t^9]]$, which is an Arf ring, we see that $\T(\m:\m)$ is finite.
\end{ex}

\section{Trace ideals of numerical semigroup rings}\label{section2}

In this section we focus on numerical semigroup rings. Throughout this section, let $H\subseteq \mathbb{N}$ be a numerical semigroup. Then, $H$ defines a local $k$-subalgebra
\[
R=k[[H]]=k[[t^h \mid h\in H]]\subseteq k[[t]],
\]
where $k[[t]]$ is the formal power series ring over a field $k$. Then $R$ satisfies the assumption written in the beginning of Section \ref{newsection2}; hence, we reuse the notation of Setup \ref{setup}.
Note that $H$ is equal to the value semigroup $v(R)$ of $R$.

\begin{thm} \label{t24}
Let $n\ge 3$.
The following conditions are equivalent.
\begin{enumerate}[\rm(1)]
\item $\mathrm{T}(R)=\I(R)$.
\item $I_iI_{i+2}=(t^{a_i})I_{i+2}$ for all $i\in\{1,\dots,n-2\}$.
\item $a_j+a_{i+1}-a_i\in H$ for all $i\in\{1,\dots,n-2\}$ and $j\in\{i+2,\dots,n\}$.
\end{enumerate}
If the residue field $k$ is infinite, then the following is also equivalent to the above conditions.
\begin{enumerate}[\rm(4)] 
\item $\T(R)$ is a finite set. 
\end{enumerate}

\end{thm}

\begin{proof}

The equivalence of (1), (2), and (4) follows by Theorem \ref{th1}.

(2) $\Rightarrow$ (3): Assume (2).
This means that $t^{-a_i}I_iI_{i+2}=I_{i+2}$ for each $i=1,\dots,n-2$.
Then, the elements $t^{a_{i+1}}\in I_{i+1}$ and $t^{a_j}\in I_{i+2}$, where $j\in\{i+2,\dots,n\}$, satisfy $t^{-a_i}t^{a_{i+1}}t^{a_j}\in I_{i+2}\subseteq R$.
It shows that $a_j+a_{i+1}-a_i\in H$.

(3)$\Rightarrow$(2): Note that the assumption (3) is equivalent to saying that $t^{a_{i+1}} I_{i+2} \subseteq (t^{a_i})$ for all $i\in\{1,\dots,n-2\}$.
We then show that for each $i\in\{1,\dots,n-2\}$, the equality $I_iI_{i+2}=t^{a_i}I_{i+2}$ holds by descending induction on $i$.
We know that the equality $I_{n-2}I_n=t^{a_{n-2}}I_n$ always holds.
Let $i<n-2$.
By the induction hypothesis, we have $I_{i+1}I_{i+3}=(t^{a_{i+1}})I_{i+3}$.
Thanks to Proposition \ref{pp} (3)$\Rightarrow$(1), we deduce the equality $I_iI_{i+2}=t^{a_i}I_{i+2}$.
\end{proof}

We also note a characterization of numerical semigroups with $\mathrm{T}(R)=\I(R)$ and $n=4$ as a special case of Theorem \ref{t24}.
In Section \ref{section6}, we consider such a situation again with paying attention to reflexive ideals.


\begin{cor}\label{cor32}
Assume $k$ is infinite and $n=4$.
Then the following conditions are equivalent.
\begin{enumerate}[\rm(1)]
\item $\mathrm{T}(R)$ is finite.
\item $\mathrm{T}(R)=\I(R)$.
\item $a_2-a_1\ge a_4-a_3$.
\end{enumerate}
\end{cor}

\begin{proof}
Since $n=4$, the condition (3) of Theorem \ref{t24} is stated as follows:
\[
a_3+a_{2}-a_1, \ \  a_4+a_2-a_1, \ \  a_4+a_{3}-a_2\in H.
\]
Since the last two of the above is larger than $a_4$, $a_4+a_2-a_1$ and $a_4+a_{3}-a_2$ are automatically in $H$. Furthermore, we have $a_3<a_3+a_{2}-a_1$. Hence, $a_3+a_{2}-a_1\in H$ if and only if $a_3+a_{2}-a_1\ge a_4$. Therefore, the assertion follows from Theorem \ref{t24}.
\end{proof}

By using Theorem \ref{t24} and Corollary \ref{cor32}, we obtain infinitely many  rings $R$ satisfying $\mathrm{T}(R)=\I(R)$ other than Arf rings (see Corollary \ref{cor29}). Since Arf rings have minimal multiplicity, we explored rings that are not of minimal multiplicity. 
Although, at least our knowledge, we are not able to describe every numerical semigroups satisfying the conditions above by giving their systems of minimal generators, we note some of them.

\begin{ex} \label{e43}
	The following numerical semigroup rings $R$ satisfy $\T(R)=\I(R)$ and $n=4$, but are not of minimal multiplicity. Let $k$ be a field.
	\begin{enumerate}[\rm(1)] 
		\item $R=k[[t^{11}, t^{14}, t^{18}, t^{20}, t^{21}, t^{23}, t^{24}, t^{26}, t^{27}, t^{30}]]$.
		\item $R=k[[t^{9}, t^{12}, t^{16}, t^{19}, t^{20}, t^{22}, t^{23}, t^{26}]]$.
		\item $R=k[[t^{5}, t^{8}, t^{12}, t^{14}]]$.
	\end{enumerate}
\end{ex}

\begin{ex}\label{ex36}
Let $n \ge 3$ be an integer, and let 
\[
H= \{ 0 \} \cup \{ 3n + 3i \in \mathbb{N} \mid \text{$0 \le i \le n -1$, but $i\ne 2$ }\} \cup \{ j\in \mathbb{N} \mid j \ge 6n \}
\]
be a numerical semigroup. Set $R=k[[H]]$. Then $\T(R)=\I(R)$ and $\ell_R(R/(R:\ol{R}))=n$. Furthermore, $R$ does not have minimal multiplicity. In particular, $R$ is not an Arf ring.
\end{ex}

\begin{proof}
We have 
\begin{align*}
a_1=3n, a_2=3n+3, a_3&=3n+9, a_4=3n+12, \dots, a_{n-1}=6n-3, \quad \text{and} \\
a_{n+k}&=6n+k \ \ \text{ for all } k\ge 0.
\end{align*}
Hence, $\ell_R(R/(R:\ol{R}))=n$. By noting that $a_{i+1}-a_i$ is either $3$ or $6$, where $i\in\{1, 2, \dots, n-3\}$, we obtain that for all $j\in\{i+2,\dots,n-1\}$, $a_j+a_{i+1}-a_i$ is either $a_j+ 3$ or $a_j+6$. In both cases, we have $a_j+a_{i+1}-a_i\in H$. It follows that $H$ satisfies Theorem \ref{t24} (3), thus $\T(R)=\I(R)$. 

Since $a_{n+6}=6n+6=2(3n+3)=2a_2$, it is straightforward to check that $R$ does not have minimal multiplicity.
\end{proof}

Let $(R, \fkm, k)$ be an analytically irreducible local domain of dimension one. In what follows, we note the relation between the conditions $\T(R)=\I (R)$ and $\T(k[[v(R)]]) = \I(k[[v(R)]])$. 

\begin{rem}\label{rem37}
Let $(R,\m,k)$ be an analytically irreducible local domain of dimension one as Section \ref{newsection2}. We reuse the notation of Setup \ref{setup}. Suppose that $\T(R)=\I(R)$. Then, $\T(k[[H]])=\I(k[[H]])$. 
\end{rem}

\begin{proof}
Since $\T(R)=\I(R)$, we have $I_iI_{i+2}=q_iI_{i+2}$ for all $i\in\{1,\dots,n-2\}$. It follows that for all $j\in\{i+2, \dots, n\}$, 
\[
q_{i+1}q_j\in I_iI_{i+2}=q_iI_{i+2} \subseteq (q_i).
\]
Hence, $q_i^{-1}q_{i+1}q_j\in R$. Thus, $-a_i+a_{i+1}+a_j\in H$. This concludes the assertion by Theorem \ref{t24}.
\end{proof}

On the other hand, the converse of the assertion in Remark \ref{rem37} does not hold in general.

\begin{ex}
Let $R=k[[t^{15}+t^{16}, t^{18}, t^{24}, t^{27}, t^{n} \mid n\ge 30]]$. Then $v(R)=\{0, 15, 18, 24, 27\}\cup \{n \mid n\ge 30\}$. Set $H=v(R)$. Note that $k[[H]]$ is the ring of Example \ref{ex36}, where $n=5$. Hence, $\T(k[[H]])=\I(k[[H]])$. On the other hand, one can obtain that $\T(R)\supsetneq \I(R)$.

Indeed, assume that $\T(R)= \I(R)$. Then, we have $I_1 I_3 = (t^{15}+t^{16})I_3$ by Theorem \ref{th1}. It follows that 
\[
t^{42}=t^{18}t^{24}\in I_1 I_3 = (t^{15}+t^{16})I_3\subseteq (t^{15}+t^{16})I_3 + t^{44}\ol{R}=(t^{39}+t^{40}, t^{42}+t^{43}) + t^{44}\ol{R}.
\]
Hence, we can write $t^{42}=f(t^{39}+t^{40}) + g(t^{42}+t^{43}) +h$, where $f, g\in R$ and $h\in t^{44}\ol{R}$. Write $f=a+f_1$ and $g=b+g_1$, where $a,b\in k$ and $f_1, g_1\in R$ with $v(f_1), v(g_1) \ge 15$. Then 
\[
t^{42}-a(t^{39}+t^{40}) - b(t^{42}+t^{43}) = f_1(t^{39}+t^{40}) + g_1(t^{42}+t^{43}) +h\in t^{44}\ol{R}.
\]
This is impossible. Hence, $I_1 I_3 \ne (t^{15}+t^{16})I_3$. It follows that $\T(R)\supsetneq \I(R)$.
\end{ex}

\section{Trace ideals over fiber products} \label{section4}

In this section, we discuss trace ideals over fiber products of local rings as a trial for the case of non-domains.
Let 
\[
R=R_1\times_k R_2
\] be a fiber product of Noetherian local rings $(R_1,\n_1,k)$ and $(R_2,\n_2,k)$ over $k$, i.e. $R$ is a subring $\{(s,t)\in R_1\times R_2\mid \pi_1(s)=\pi_2(t)\}$ of a usual product $R_1\times R_2$, where $\pi_1\colon R_1 \to k$ and $\pi_2\colon R_2\to k$ are canonical surjections.
Let $\fkm$ denote the maximal ideal of $R$. 
The canonical maps $p_1\colon R\to R_1$ and $p_2\colon R\to R_2$ are surjective homomorphisms of rings. In addition, there are isomorphisms 
\begin{center}
	$i_1\colon\n_1 \cong \Ker p_2 = \fkn_1 \times (0)$ \ \  and \ \  $i_2\colon\n_2 \cong \Ker p_1=(0) \times \fkn_2$
\end{center} 
as $R$-modules.
And $\m$ has a decomposition $\m=\Ker p_2\oplus \Ker p_1$ as an $R$-module.

\begin{thm} \label{th51}
Let $(R_1,\n_1,k)$ and $(R_2,\n_2,k)$ be (not necessarily one--dimensional Cohen--Macaulay) local rings with positive depth. Let $R$ be a fiber product $R_1\times_k R_2$ of $R_1$ and $R_2$ over $k$.
Then 
\[
\T(R)=\{i_1(I)\oplus i_2(J)\mid I\in X_1,\ J\in X_2\}\cup\{R\},
\]
where $X_1$ and $X_2$ are defined as follows:
\begin{enumerate}[\rm(1)]
\item If $R_1$ (resp. $R_2$) is a discrete valuation ring, then $X_1=\{\fkn_1\}$ (resp. $X_2=\{\fkn_2\})$.
\item If $R_1$ (resp. $R_2$) is not a discrete valuation ring, then $X_1=\T(R_1)\setminus\{R_1\}$ (resp. $X_2=\T(R_2)\setminus\{R_2\}$).
\end{enumerate}
\end{thm}

\begin{proof}
($\subseteq$): Let $L$ be an ideal in $\T(R)$ with $I\not=R$.
Then one has an equality $L=i_2p_2(L)\oplus i_1p_1(L)$.
Indeed, the inclusion $L\subseteq i_2p_2(L)\oplus i_1p_1(L)$ is clear.
Since there are surjections $i_1p_1\colon L \to i_1p_1(L)$ and $i_2p_2\colon L \to i_2p_2(L)$, it yields that $i_2p_2(L)\oplus i_1p_1(L)\subseteq \tr_R(L)=L$.
Thus we only need to know what $p_1(L)$ and $p_2(L)$ are. Note that both $p_1(L)$ and $p_2(L)$ are nonzero. Indeed, if $p_1(L)=0$, then $L=i_2p_2(L)$ is annihilated by $i_1(n_1)$. This means that $L$ is not a regular ideal of $R$.

(1): If $R_1$ is a discrete valuation ring, then $p_1(L)$ is isomorphic to $R_1\cong \n_1$. Thus, we have a surjection $L\to i_1(\n_1)(\subseteq \m)$.
Therefore, $i_1(\n_1)$ is contained in $\tr_R(L)=L$.
In particular, one obtains $\n_1 \supseteq p_1(L) \supseteq p_1(i_1(\n_1))=\n_1$.

(2): Suppose that $R_1$ is not a discrete valuation ring.
What we need to prove is that $p_1(L)$ belongs to $\T(R_1)\setminus \{R_1\}$.
In order to show this, let $f\colon p_1(L)\to R_1$ be a homomorphism of modules.
Assume that $\Im f=R_1$. Then, since there exists a surjection $R_1^{\oplus a} \to \n_1$ for some integer $a>0$, we obtain the surjective homomorphism $L^{\oplus a} \to R_1^{\oplus a} \to \n_1$. 
Thus, $i_1(\n_1)$ is contained in $\tr_R(L)(=L)$, which yields that $p_1(L)= \n_1$.
It follows that $f$ induces a surjection $\n_1 \to R_1$; hence, $R_1$ is a discrete valuation ring.
This contradicts our assumption.
We now see that an inclusion $\Im f\subseteq \n_1$ holds for any homomorphism $f\in\Hom_{R_1}(p_1(L), R_1)$. 
Take the composition $i_1 f p_1 \colon L \to R$.
We have $\Im(i_1 f p_1)\subseteq \tr_R(L)=L$.
Hence, we obtain that $p_1(L) \supseteq \Im (p_1 i_1 f p_1) = \Im (p_1 i_1 f) =\Im f$.
This means that $p_1(L)$ is a trace ideal of $R_1$.

($\supseteq$): Let $L=i_1(I)\oplus i_2(J)$, where $I\in X_1$ and $J\in X_2$. Then, since $\mu_R(L)=\mu_R(i_1(I)) + \mu_R(i_2(J))>1$, $L$ has no free summands. Hence, 
\[
\Hom_R(L, R) = \Hom_R(L, \fkm)=\Hom_R(i_1(I)\oplus i_2(J), i_1(\fkn_1) \oplus i_2(\fkn_2)).
\]
Assume that $f\in \Hom_R(i_1(I), i_2(\fkn_2))$. Then, since $i_1(I)$ is annihilated by $i_2(\fkn_2)$, $\Im f$ is also annihilated by $i_2(\fkn_2)$. By noting that $\depth R_2>0$, it follows that $f=0$. By the same argument, we have $\Hom_R(i_2(J), i_1(\fkn_1))=0$. Hence, 
\[
\Hom_R(L, \fkm)=\Hom_R(i_1(I), i_1(\fkn_1)) \oplus \Hom_R(i_2(J), i_2(\fkn_2)). 
\]
Therefore, it is enough to prove that 
\begin{equation}\label{eqn411}
	\begin{split}
	\Hom_R(i_1(I), i_1(\fkn_1)) &= \Hom_R(i_1(I), i_1(I)) \quad \text{and} \\ 
	\Hom_R(i_2(J), i_2(\fkn_2)) &= \Hom_R(i_2(J), i_2(J)).
	\end{split}
\end{equation} 
Indeed, \eqref{eqn411} shows that $\Hom_R(L, \fkm)=\Hom_R(L, L)$.

If $R_1$ (resp. $R_2$) is a discrete valuation ring, then $I=\fkn_1$ (resp. $J=\fkn_2$). Hence, \eqref{eqn411} holds. 
If $R_1$ (resp. $R_2$) is not a discrete valuation ring, then $I\in \T(R_1)\setminus\{R_1\}$ (resp. $J\in \T(R_2)\setminus\{R_2\}$). In any case, \eqref{eqn411} holds. 
This completes the proof.
\end{proof}

\begin{cor}
Let $R$ be a fiber product $R_1\times_k R_2$ of local rings $(R_1,\n_1,k)$ and $(R_2,\n_2,k)$ with positive depth over $k$.
Then $\T(R)$ is finite if and only if so are both $\T(R_1)$ and $\T(R_2)$.
\end{cor}

\begin{ex}
	Let $R$ be a fiber product 
	\[
	k[[t^{5}, t^{8}, t^{12}, t^{14}]]\times_k k[[t^{9}, t^{12}, t^{16}, t^{19}, t^{20}, t^{22}, t^{23}, t^{26}]].
	\]
	Then, 
it is clear that $R$ is not a domain. On the other hand, since both $\T(k[[t^{5}, t^{8}, t^{12}, t^{14}]])$ and $\T(k[[t^{9}, t^{12}, t^{16}, t^{19}, t^{20}, t^{22}, t^{23}, t^{26}]])$ are finite by Example \ref{e43}, $\T(R)$ is also finite.
\end{ex}


\section{Some special reflexive modules}\label{section5}

Throughout this section, we employ Setup \ref{setup}.
Denote by $\mathrm{Ref}_1(R)$ the set of isomorphism classes of reflexive modules of rank one over $R$.
We say a fractional ideal $I$ is \textit{reflexive} if $R:(R:I)=I$.
Note that an ideal $I$ is reflexive exactly when its isomorphism class belongs to $\mathrm{Ref}_1(R)$.

As a first part of this section, we prove that $\mathrm{Ref}_1(R)$ is finite when the equality $\T(R)=\I(R)$ holds.

\begin{lem}\label{3.1}
Let $M$ be a reflexive $R$-module of rank one.
Then there exists a reflexive ideal $I$ of $R$ such that $I$ is isomorphic to $M$ and contains $\c$.
\end{lem}

\begin{proof}
First note that $M$ is isomorphic to some nonzero ideal $J$ of $R$. Set $a_i=\min\{v(x)\mid~x\in J\}$ and take an element $q\in J$ such that $v(q)=a_i$.
Then both of the integral closures of $J$ and $(q)$ are equal to $I_i$.
Hence $(q)$ is a minimal reduction of $J$, that is, $J^{\ell+1}=qJ^\ell$ for some $\ell>0$. 
By \cite[Theorem 3.5]{dms}, $J$ is isomorphic to an ideal $I$ containing $\fkc$.
As $M\cong I$, it is clear that $I$ is reflexive.
\end{proof}

\begin{thm}\label{cor3.4}
Assume $\mathrm{T}(R)=\I(R)$.
Then there is an inclusion map from $\mathrm{Ref}_1(R)$ to $\I(R)\cup\{J_0^{(i)}\}_{i\in\{1,\dots,n-2\}}$.
In particular, $\mathrm{Ref}_1(R)$ is a finite set.
\end{thm}

\begin{proof}
Let $I$ be a reflexive ideal of $R$ containing $\fkc$.
Take an integer $i$ and an element $q\in I$ such that $a_i=v(q)=\min\{v(f)\mid f\in I\}$.
The inequality $i\le n$ is obvious.
\begin{claim} \label{c621}

We have either one of the following:
\begin{enumerate}[\rm(1)] 
	\item $I=I_i$ or 
	\item $i\le n-2$ and $I=J^{(i)}_\alpha$ for some $\alpha\in R$, where $J_\alpha^{(i)}$ denotes the ideal defined after Lemma \ref{lem2.5}.
\end{enumerate}
\end{claim}

\begin{claim} \label{c622}
If $i\le n-2$ and $I=J_\alpha^{(i)}$ for some $\alpha\in R$, then $I$ is isomorphic to $J_0^{(i)}$.
\end{claim}

\begin{proof}[Proof of Claim \ref{c621}]
The case where $n-2 \le i\le n$ is clear since $I$ contains $\fkc$.
So we may assume $i\le n-3$.
Since $I$ contains $\c$, $R:I\subseteq R:\c=\ol{R}$.
Then observe that $q(R:I)\subseteq q\ol{R}\cap R\subseteq\{x\in R\mid v(x)\ge v(q)\}=I_i$.
Therefore $I=R:(R:I)\supseteq R:q^{-1}I_i$.
Using Theorem \ref{t24} and the assumption $\mathrm{T}(R)=\I(R)$, we see that $q^{-1}I_iI_{i+2}\subseteq I_{i+2}\subseteq R$.
It follows that $R:q^{-1}I_i\supseteq I_{i+2}$.
We then have an inclusion $I\supseteq I_{i+2}$, which yields that either $I=I_i$ or $J_{\alpha}^{(i)}$ for some $\alpha \in R$.
\end{proof}
\begin{proof}[Proof of Claim \ref{c622}]
We set $x:=1+\alpha q^{-1}q'$, where $q'$ is an element taken as in the definition of $J_{\alpha}^{(i)}$.
Then $qx=q+\alpha q'$.
In view of Theorem \ref{t24}, the assumption $\mathrm{T}(R)=\I(R)$ implies $xI_{i+2}=I_{i+2}$. Indeed, $xI_{i+2}\subseteq I_{i+2}$ follows from $xI_{i+2}\subseteq I_{i+2}+ \alpha q^{-1}q'I_{i+2}$ and $\alpha q^{-1}q'I_{i+2}\subseteq \alpha q^{-1}I_iI_{i+2} \subseteq I_{i+2}$. On the other hand, the inclusion $xI_{i+2}\supseteq I_{i+2}$ follows from the observation that $xI_{i+2}$ contains $\fkc$ and all elements of order $a_j$ for $i \le j \le n$ since $v(x)=0$.

Thus we get $xJ_0^{(i)}= (xq)+xI_{i+2}=(q+\alpha q') + I_{i+2}= J_\alpha^{(i)}$.
This means that $J_0^{(i)}$ is isomorphic to $J_\alpha^{(i)}$ via the multiplication by $x$.
\end{proof}
By Claims \ref{c621} and \ref{c622}, reflexive ideals containing $\fkc$ are only either $I_j$ for $0\le j \le n$ or $J_0^{(i)}$ for $0 \le i \le n-2$ up to isomorphism. By combining this result with Lemma \ref{3.1}, a system of representatives of $\Ref_1(R)$ is a subset of $\I(R)\cup\{J_0^{(i)}\}_{i\in\{1,\dots,n-2\}}$.
\end{proof}

Next we explore reflexive Ulrich modules over rings $R$ satisfying an equality $\m I_3=q I_3$ for some $q\in \m$.
Note that rings $R$ satisfying $\T(R)=\I(R)$ have the equality $\m I_3=q I_3$ (Theorem \ref{th1}).
Let us recall the notion of Ulrich modules.

\begin{dfn}(\cite[Definition 3.1]{GOTWY})
We say that a finitely generated $R$-module $M$ is an {\it Ulrich module} if $M$ is maximal Cohen--Macaulay (equivalently, torsion--free since $\dim R=\depth R=1$), and $\rme(M)=\mu_R(M)$, where $\rme(M)$ denotes the multiplicity of $M$ and $\mu_R(M)$ denotes the number of minimal generators of $M$. It is known that $M$ is Ulrich module if and only if $\m M=q M$, where $(q)$ is a minimal reduction of $\m$ (see \cite{GOTWY}).
\end{dfn}

In what follows, throughout this section, let $(q)$ be a minimal reduction of $\m$.

\begin{lem}{\rm (\cite{b})}\label{prelim3.13}
Let $M$ be a finitely generated reflexive $R$-module such that $M$ has no free summands. Then, $M$ can be regarded as an $\m:\m$-module. That is, by regarding $M$ as a submodule of $Q(R)\otimes_R M\cong Q(R)^{\rank_R(M)}$, we have $(\m:\m)M=M$.
\end{lem}

\begin{lem}\label{lem2.13}
Let $M$ be an Ulrich $R$-module. Then $\Hom_R(M, R)$ is a reflexive Ulrich $R$-module.
\end{lem}

\begin{proof}
By applying the $R$-dual to $0 \to M \xrightarrow{q} M \to M/q M \to 0$, we obtain an exact sequence
\[
0 \to \Hom_R(M, R) \xrightarrow{q} \Hom_R(M, R) \to \Ext_R^1(M/qM, R).
\]
Note that $\Ext_R^1(M/qM, R)$ is a free $R/\m$-module since $\m M=q M$.
Hence, the above exact sequence proves that $\Hom_R(M, R)/q \Hom_R(M, R)$ is a free $R/(q)$-module.
It follows that $\m \Hom_R(M, R) \subseteq q \Hom_R(M, R)$. Hence, $\Hom_R(M, R)$ is a Ulrich $R$-module.
The reflexivity of $\Hom_R(M, R)$ follows from a well-known fact, see \cite[Lemma 4.1]{HKS} for example.
\end{proof}

\begin{lem}\label{lem2.14}
Set $S=\m:\m$.
If $M$ is a reflexive Ulrich $R$-module, then $M$ is a reflexive $S$-module.
\end{lem}

\begin{proof}
By Lemma \ref{prelim3.13}, $M$ can be regarded as an $S$-module.
Let $X$ be the kernel of the canonical surjective $S$-homomorphism $\m \otimes_S M \to \m M;\  a\otimes x \mapsto ax$ for $a\in \m$ and $x\in M$.
Note that $X$ is of finite length as an $R$-module since there are equalities \[
\rank_R (X)=\rank_R (\m \otimes_S M)- \rank_R (\m M)=0.
\]
Hence, by applying the $R$-dual to $0 \to X \to \m \otimes_S M \to \m M \to 0$, we obtain an isomorphism 
$
\Hom_R(M\otimes_S \m, R) \cong \Hom_R(\m M, R)$.
Therefore, we obtain that 
\begin{align*}
\Hom_S(M, S) =& \Hom_S(M, \Hom_R(\m, R)) \cong \Hom_R(M\otimes_S \m, R) \cong \Hom_R(\m M, R)\\
=& \Hom_R(q M, R) \cong \Hom_R(M, R).
\end{align*}
By noting that $\Hom_R(M, R)$ is again a reflexive Ulrich $R$-module by Lemma \ref{lem2.13}, we obtain that 
\[
\Hom_S(\Hom_S(M, S), S) \cong \Hom_S(\Hom_R(M, R), S) \cong \Hom_R(\Hom_R(M, R), R) \cong M.
\]
Hence, $M$ is reflexive as an $S$-module (\cite[Lemma 4.1]{HKS}).
\end{proof}

We now characterize reflexive Ulrich $R$-modules in terms of the endomorphism algebra $\m:\m$ of $\m$.

\begin{thm} \label{th59}
Suppose that an equality
$\m I_3=q I_3$ holds.
Set $S=\m:\m$. Let $M$ be a finitely generated $R$-module such that $R$ and $S$ are not in the direct summand of $M$.
Then, the following are equivalent.
\begin{enumerate}[\rm(1)] 
\item $M$ is a reflexive Ulrich $R$-module.
\item $M$ is a reflexive $S$-module.
\end{enumerate}
\end{thm}

\begin{proof}

(1) $\Rightarrow$ (2): This follows by Lemma \ref{lem2.14}.

(2)  $\Rightarrow$ (1): Suppose that $M$ is a reflexive $S$-module. Then $M$ is reflexive as an $R$-module by \cite[Theorem 1.3(1)]{ko}. Thus, we have only to show that $M$ is an Ulrich $R$-module.

Let $\fkn$ be the maximal ideal of $S$.
Since $S$ is not in the direct summand of $M$, $M$ can be regarded as an $\fkn:\fkn$-module by Lemma \ref{prelim3.13}. 
Suppose that $\m^2\ne q\m$. Then, by Lemma \ref{lem2.10}, $\fkn:\fkn=q^{-1}I_3:q^{-1}I_3=I_3:I_3$. Hence, by Theorem \ref{t24}, we have $q^{-1}I_1M \subseteq (I_3:I_3)M=M$. Hence, $M$ is an Ulrich $R$-module. 

Suppose that $\m^2=q\m$. Then, by Lemma \ref{lem2.9}, $q^{-1}\m=S$. Hence, $q^{-1}\m M=SM=M$, that is, $M$ is an Ulrich $R$-module. 
\end{proof}

As an application, we obtain the finiteness of reflexive Ulrich $R$-modules up to isomorphism when $n$ is small.
Before showing it, we put a lemma.

\begin{lem}\label{rem5.10}
Suppose that $R$ is not a discrete valuation ring. 
Let $S=\fkm:\fkm$ and $\fkc_S$ denote the conductor of $S$. Then, 
$\ell_S(S/\fkc_S) < \ell_R(R/\fkc)$. Furthermore, $\ell_S(S/\fkc_S) = \ell_R(R/\fkc)-1$ if and only if $R$ has minimal multiplicity.
\end{lem}

\begin{proof}
Note that $\fkc_S=S:\ol{S}=(R:\fkm): \ol{R}=R:\fkm\ol{R}$. Therefore, by noting that $\fkm\ol{R}=q \ol{R}$, we obtain that $\fkc_S=R:q\ol{R}=q^{-1} \fkc$. It follows that $\ell_S(S/\fkc_S) =\ell_R(S/\fkc_S)=\ell_R(qS/\fkc)\le \ell_R(\fkm/\fkc)=\ell_R(R/\fkc)-1$, where the third inequality follows from $qS\subseteq \fkm$.

The equality $\ell_S(S/\fkc_S) = \ell_R(R/\fkc)-1$ is equivalent to saying that $qS=\fkm$. This is also equivalent to saying that $\fkm^2=q\fkm$ by Lemma \ref{lem2.9}.
\end{proof}

\begin{cor}
Assume that either of the following holds:
\begin{enumerate}[\rm(1)]
\item $n\le 3$.
\item $n=4$, $\m I_3=q I_3$, and $R$ is not of minimal multiplicity.
\end{enumerate}
Then there exist only finitely many reflexive Ulrich $R$-modules up to isomorphism.
\end{cor}

\begin{proof}
Set $S=\m:\m$.
By Theorem \ref{th59}, it is enough to show that there exist only finitely many reflexive $S$-modules up to isomorphism.

By Lemma \ref{rem5.10}, $\ell_S(S/\c_S)\le 2$, where $\c_S$ is the conductor of $S$. 
Then, by Lemma \ref{lem2.5}, $S$ has minimal multiplicity. Let $\fkn$ be the maximal ideal of $S$, and set $S_1=\fkn:\fkn$. Then, $\ell_{S_1} (S_1/\fkc_{S_1}) \le 1$. It follows that $S_1$ again has  minimal multiplicity by Lemma \ref{lem2.5}. Therefore, $S_1$ or the endomorphism algebra of the maximal ideal of $S_1$ is a discrete valuation ring.
In any case, we obtain that $S$ is an Arf ring by \cite{Lipman}.

In particular, there exist only finitely many reflexive $S$-modules up to isomorphism by \cite[Corollary 3.6]{ik}.
\end{proof}

\section{Reflexive ideals in numerical semigroup rings with small non-gaps} \label{section6}

The purpose of this section is to explore the relation between the finiteness of $\mathrm{T}(R)$ and that of $\mathrm{Ref}_1(R)$. We maintain the notations of Section \ref{section2}. We already saw that both $\mathrm{T}(R)$ and $\mathrm{Ref}_1(R)$ are finite if $n=\ell_R(R/\fkc)\le 3$ (Corollary \ref{cor23} and \cite[Theorem 6.8]{dms}). Thus, we focus on the case of $n=4$. The goal of this section is to prove Theorem \ref{thm4.1}. Let us prepare notations to describe Theorem \ref{thm4.1}. We say that an ideal $I$ is {\it monomial} if $I$ is generated by monomial elements. Set 
\[
\RT(R) = \{ I\in \T(R) \mid \text{$I$ is reflexive}\}.
\]


\begin{thm}\label{thm4.1}
Suppose that $n=4$ and $k$ is infinite. Then the following conditions are equivalent.
\begin{enumerate}[\rm(1)] 
\item For all $I\in \Ref_1(R)$, $I$ is isomorphic to some monomial ideal containing $\fkc$.
\item $\Ref_1(R)$ is finite.
\item $\RT(R)$ is finite.
\item Either one of the following holds true:
\begin{enumerate}[\rm(i)] 
\item $a_2-a_1+a_3\ge a_4$, that is, $\T(R)$ is finite.
\item $2a_3 - a_1 < a_4$.
\end{enumerate}
\end{enumerate}
\end{thm}

To prove Theorem \ref{thm4.1}, we note several lemmas.

\begin{lem}\label{rem4.2}
Let $I$ be an ideal of $R$ containing $\fkc$. Then, $R:I\subseteq \ol{R}$.
\end{lem}

\begin{proof}
Since $\fkc \subseteq I$, we obtain that $R:I \subseteq R:\fkc=R:t^{a_n}\ol{R}=t^{-a_n}(R:\ol{R})=\ol{R}$.
\end{proof}

\begin{lem}\label{lem4.2}
Let $I=(f) + \fkc$ be an ideal of $R$, where $f\in R$. Then $I\cong (t^{v(f)}) + \fkc$.
\end{lem}

\begin{proof}
$f$ can be written in the form $t^{v(f)} + t^{v(f)}x$, where $x\in \ol{R}$ with $v(x)\ge 1$. Hence, $\fkc=(1+x)\fkc$ and 
\[
(t^{v(f)}) + \fkc \cong (1+x) [(t^{v(f)}) + \fkc] = (f) + (1+x)\fkc = I.
\]
\end{proof}

\begin{lem}\label{lem4.3}
Let $I$ be an ideal of $R$. Let $a_i=\min \{v(f) \in H \mid f\in I\}$. Then 
\begin{enumerate}[\rm(1)] 
\item $R + t^{a_n - a_i} \ol{R} \subseteq R:I$.
\item $R:[R + t^{a_n - a_i} \ol{R}]=I_i$. Hence, $I\subseteq R:(R:I)\subseteq I_i$.
\end{enumerate}
\end{lem}

\begin{proof}
(1): $R \subseteq R:I$ is trivial. Note that $t^{a_n - a_i} I \subseteq t^{a_n} \ol{R} = \fkc$ since $v(f)\ge a_i$ for all $f\in I$. Hence, $t^{a_n - a_i} I \ol{R}\subseteq R$, that is, $t^{a_n - a_i} \ol{R} \subseteq R:I$.

(2): Note that $R:[R + t^{a_n - a_i} \ol{R}]=(R:R)\cap (R:t^{a_n - a_i} \ol{R})$. On the other hand, we obtain that 
\[
R:t^{a_n - a_i} \ol{R}=t^{a_i - a_n} (R:\ol{R})=t^{a_i - a_n} t^{a_n} \ol{R}=t^{a_i} \ol{R}.
\]
Hence, $R:[R + t^{a_n - a_i} \ol{R}]=R\cap t^{a_i} \ol{R}=I_i$. Therefore, by (1), we obtain that $R:(R:I) \subseteq R:[R + t^{a_n - a_i} \ol{R}]=I_i$. 
\end{proof}

\begin{lem} \label{l65}
Assume that $n=4$ and $a_2-a_1+a_3\not\in H$.
Then the following hold true.
\begin{enumerate}[\rm(1)]
\item For each $\alpha \in R$, $J_\alpha^{(1)}:=(t^{a_1}+\alpha t^{a_2})+I_3\in \T(R)$.
\item Let $\alpha, \beta\in k$. If $\alpha \ne \beta$, then $J_\alpha^{(1)} \ne J_\beta^{(1)}$.
\end{enumerate}
\end{lem}

\begin{proof}
Since $n=4$, the equality $I_2I_4=t^{a_2}I_4$ holds.
On the other hand, the inequality $I_1I_3\not=t^{a_1}I_3$ follows by the assumption $a_2-a_1+a_3\not\in H$.
So we may apply Proposition \ref{ppp}.
\end{proof}

Now we prove Theorem \ref{thm4.1}.

\begin{proof}[Proof of Theorem \ref{thm4.1}]
(1)$\Rightarrow$(2): This is clear. 

(2)$\Rightarrow$(3): Recall that for $I, J\in \T(R)$, $I=J$ if $I\cong J$; see \cite[Corollary 1.2(a)]{hr} for example. Hence, we can regard $\RT(R)$ as a subset of $\Ref_1(R)$. Thus, (2)$\Rightarrow$(3) holds.

(3)$\Rightarrow$(4): Suppose that $a_2-a_1+a_3< a_4$ and $2a_3 - a_1 \ge a_4$. It is enough to prove that $\RT(R)$ is infinite. Let $\alpha\in k$ and $I=(t^{a_1} + \alpha t^{a_2}, t^{a_3}) + \fkc$. Then, it is enough to show that $I\in \Ref_1(R)$. Indeed, by noting that $a_3<a_2-a_1+a_3< a_4$ implies that $a_2-a_1+a_3\not\in H$, we have  $I=J_{\alpha}^{(1)}\in \T(R)$ by Lemma \ref{l65}(1). We further prove that $I=J_{\alpha}^{(1)}$ is a reflexive ideal for each $\alpha\in k$. Then, we complete the proof since $\RT(R)$ is infinite by Lemma \ref{l65}(2).

Set $f=t^{a_1} + \alpha t^{a_2}$ and $x=-\alpha t^{a_2-a_1}$. Then $f=t^{a_1} (1-x)$. Set $g=t^{a_3-a_1}(1+x+\cdots +x^\ell)$, where $\ell\ge a_4$. We obtain that 
\begin{align*}
fg=t^{a_3}(1-x^{\ell+1}), \quad t^{a_3}g=t^{2a_3-a_1}(1+x+\cdots +x^\ell), \quad \text{and} \quad g\fkc\subseteq \fkc.
\end{align*}
Since we assume that $2a_3 - a_1 \ge a_4$, it follows that $g\in R:I$. By Lemma \ref{lem4.3}(1), $R + t^{a_n - a_1} \ol{R} + (g) \subseteq R:I$. Hence, $R:(R:I) \subseteq R:[R + t^{a_n - a_1} \ol{R} + (g)]=I_1 \cap (R:g)$ by Lemma \ref{lem4.3}(2). Let $h\in I_1 \cap (R:g)$. We can write $h=d_1t^{a_1} + d_2t^{a_2} + d_3t^{a_3} + \cdots$, where $d_i\in k$. Then, 
\begin{align*}
gh &\equiv t^{a_3-a_1}(1+x+\cdots +x^\ell) (d_1t^{a_1} + d_2t^{a_2} + d_3t^{a_3}) & (\mod \fkc)\\
&\equiv (1+x+\cdots +x^\ell)(d_1t^{a_3} + d_2t^{a_2+a_3-a_1} ) & (\mod \fkc).
\end{align*}
By noting that $v(x)=a_2-a_1$, the $gh$'s coefficient of degree $a_2+a_3-a_1$ is $-\alpha d_1 +d_2$. On the other hand, we have $gh\in R$ and $a_2-a_1+a_3\not\in H$. Hence, we obtain that $-\alpha d_1 +d_2=0$. It follows that 
\[
h= d_1(t^{a_1} + \alpha t^{a_2}) + d_3t^{a_3} + \cdots \quad  \in (t^{a_1} + \alpha t^{a_2}, t^{a_3}) + \fkc=I.
\]
Hence, $I_1 \cap (R:g)\subseteq I$. In conclusion, we obtain that $I\subseteq R:(R:I)\subseteq I_1 \cap (R:g)\subseteq I$. Hence $I$ is a reflexive ideal.

(4)(i)$\Rightarrow$(1): This follows from Theorem \ref{cor3.4}.

(4)(ii)$\Rightarrow$(1): Suppose that $I$ is a reflexive ideal. By Lemma \ref{3.1}, we may assume that $\fkc \subseteq I$. Then $I$ forms one of the following. Let $\alpha, \beta \in k$.
\begin{enumerate}[\rm(a)] 
\item $I=I_0, I_1, I_2, I_3, I_4$.
\item $I=(t^{a_2} + \alpha t^{a_3}) + \fkc$.
\item $I=(t^{a_1} + \alpha t^{a_2} + \beta t^{a_3}) + \fkc$.
\item $I=(t^{a_1} + \alpha t^{a_2}, t^{a_3}) + \fkc$.
\item $I=(t^{a_1} + \alpha t^{a_3}, t^{a_2} + \beta t^{a_3}) + \fkc$.
\end{enumerate}

For the case (a), there is nothing to prove. By Lemma \ref{lem4.2}, in the cases (b) and (c), $I$ is isomorphic to some monomial ideal containing $\fkc$. Thus, it is enough to prove the following claims:

\begin{claim}\label{claim2}
Suppose that $2a_3 - a_1 < a_4$. Let $I=(t^{a_1} + \alpha t^{a_2}, t^{a_3}) + \fkc$. Then $R:(R:I)=I_1$. 
\end{claim}

\begin{claim}\label{claim3}
Suppose that $2a_3 - a_1 < a_4$. Let $I=(t^{a_1} + \alpha t^{a_3}, t^{a_2} + \beta t^{a_3}) + \fkc$. Then the following hold true.
\begin{enumerate}[\rm(d-1)] 
\item If $a_1 + a_3 \ne 2a_2$, then $R:(R:I)=I_1$.
\item If $a_1 + a_3 = 2a_2$ and $\alpha\ne -\beta^2$, then $R:(R:I)=I_1$.
\item If $a_1 + a_3 = 2a_2$ and $\alpha = -\beta^2$, then $I\cong (t^{a_1}, t^{a_2}) + \fkc$.
\end{enumerate}
\end{claim}

\begin{proof}[Proof of Claim \ref{claim2}]
It is enough to prove that $R:I \subseteq R + t^{a_4 - a_1} \ol{R}$. Indeed, if $R:I \subseteq R + t^{a_4 - a_1} \ol{R}$, then we have $R:I = R + t^{a_4 - a_1} \ol{R}$ by Lemma \ref{lem4.3}(1). Hence, $R:(R:I)=I_1$ by Lemma \ref{lem4.3}(2).

Let $g\in R:I$. Then, by Lemma \ref{rem4.2}, we can write $g=c_0 + g'$, where $c_0\in k$ and $g'\in R:I$ such that $v(g')>0$. Then $g'(t^{a_1} + \alpha t^{a_2})\in R$ and $g't^{a_3}\in R$ since $g'I\subseteq R$. This proves that 
\[
v(g') + a_1\in H \quad \text{and} \quad v(g') + a_3\in H.
\]
Hence, we have $v(g')+a_1=a_i$ for some $i\ge 2$, and $a_i - a_1 + a_3\in H$.
On the other hand, by the assumption, we have $a_3 < a_3 + a_2 - a_1<2a_3 - a_1 < a_4$. Thus, $a_3 + a_2 - a_1, 2a_3 - a_1\not\in H$. This proves that $i\ne 2,3$. Therefore, $v(g')\ge a_4 - a_1$, that is, $g'\in t^{a_4 - a_1}\ol{R}$. It follows that $g=c_0 + g' \in R + t^{a_4 - a_1} \ol{R}$.
\end{proof}

\begin{proof}[Proof of Claim \ref{claim3}]
(d-1): This proof proceeds in the same way as the proof of Claim \ref{claim2}. 
As we explain in the beginning of the proof of Claim \ref{claim2}, it is enough to prove that $R:I \subseteq R + t^{a_4 - a_1} \ol{R}$. Let $g\in R:I$ and write $g=c_0 + g'$, where $c_0\in k$ and $g'\in R:I$ such that $v(g')>0$. 
Then $g'(t^{a_1} + \alpha t^{a_3})\in R$ and $g'(t^{a_2} + \beta t^{a_3})\in R$ since $g'I\subseteq R$. This proves that 
\[
v(g') + a_1\in H \quad \text{and} \quad v(g') + a_2\in H.
\]
Hence, we have $v(g')+a_1=a_i$ for some $i\ge 2$, and $a_i - a_1 + a_2=a_j$ for some $j\ge 3$. We show that $i\ge 4$. Assume that $i=2$. Then $2a_2 - a_1=a_j$ for some $j\ge 3$. By the assumption of (d-1), we obtain that $j\ne 3$. But, because $2a_2-a_1<2a_3 - a_1 < a_4$, $j\ge 4$ is also impossible. Thus, $i\ne 2$.
Assume that $i=3$. Then $a_3 - a_1 + a_2=a_j$ for some $j\ge 3$. Since $a_3<a_3 - a_1 + a_2$, $j\ne 3$. It follows that $a_3 - a_1 + a_2\ge a_4$. This contradicts for the assumption $2a_3 - a_1 < a_4$. Therefore, $i\ge 4$. 
It follows that $v(g')\ge a_4 - a_1$, that is, $g'\in t^{a_4 - a_1}\ol{R}$. Hence, $g=c_0 + g' \in R + t^{a_n - a_1} \ol{R}$.

(d-2): Set $s=a_2-a_1$. By the assumptions, $a_3=2a_2-a_1=a_1 +2s$ and $a_3+2s=2a_3-a_1<a_4$. 
Hence, we obtain that 
\begin{align}\label{Hform}
a_2=a_1 + s, \quad a_3=a_1 +2s, \quad \text{and} \quad a_4-a_3\ge 2s+1.
\end{align}
Set $f_1=t^{a_1} + \alpha t^{a_1+2s}$ and $f_2=t^{a_1+s} + \beta t^{a_1+2s}$. Then $R:I=(R:f_1)\cap (R:f_2) \cap \ol{R}$ by Lemma \ref{rem4.2}. Let $g\in R:I$, and write $g=c_0+c_1t + c_2 t^2 + \cdots$, where $c_i\in k$. Then,
for all $x\ge a_1 + 2s$, we obtain that 
\begin{align}\label{coef}
\begin{split}
\text{(the  $f_1g$'s coefficient of degree $x$)} &= \text{$c_{x-a_1} + \alpha c_{x-(a_1+2s)}$}\\ 
\text{(the  $f_2g$'s coefficient of degree $x$)} &= \text{$c_{x-(a_1+s)} + \beta c_{x-(a_1+2s)}$}.
\end{split}
\end{align}

Here, suppose that $x, x+s\not\in H$. By \eqref{coef}, we obtain that
\begin{align}\label{eq1}
c_{x-a_1} + \alpha c_{x-(a_1+2s)}&=0\\ \label{eq2}
c_{x+s-a_1} + \alpha c_{x+s-(a_1+2s)}&=0\\ \label{eq3}
c_{x-(a_1+s)} + \beta c_{x-(a_1+2s)}&=0\\ \label{eq4}
c_{x+s-(a_1+s)} + \beta c_{x+s-(a_1+2s)}&=0.
\end{align}
By \eqref{eq1}, \eqref{eq4}, and \eqref{eq3}, we have 
\begin{align}\label{eq5}
-\alpha c_{x-a_1-2s}=c_{x-a_1} = -\beta c_{x-a_1-s} = \beta^2 c_{x-a_1-2s}.
\end{align}
Therefore, since we assume that $\alpha\ne -\beta^2$, we obtain that $c_{x-a_1-2s}=0$. It follows that 
\begin{align}\label{eqfinal}
c_{x-a_1+s}=c_{x-a_1}=c_{x-a_1-s}=c_{x-a_1-2s}=0
\end{align}
by \eqref{eq1}-\eqref{eq4}. That is, if $x, x+s\not\in H$, then we have \eqref{eqfinal}. 

On the other hand, $x, x+s\not\in H$ holds for all $a_3+1\le x \le a_4-s-1$. Note that the number of (consecutive) integers between $a_3+1$ and $a_4-s-1$ is $a_4-s-1-a_3\ge s$ by \eqref{Hform}. Therefore, the fact that \eqref{eqfinal} holds for all $x=a_3+1, \dots, a_4-s-1$ turns out that 
\[
c_{(a_3+1)-a_1-2s}=\cdots =c_{(a_4-s-1)-a_1+s}=0.
\]
By noting that $(a_3+1)-a_1-2s=1$ and $(a_4-s-1)-a_1+s=a_4-a_1-1$ due to \eqref{Hform}, we obtain that 
\[
g=c_0 + c_{a_4-a_1} t^{a_4-a_1} + c_{a_4-a_1+1} t^{a_4-a_1+1} + \cdots \in R+t^{a_4-a_1}\ol{R}.
\]
Therefore, by combining this result with Lemma \ref{lem4.2}, $R:(R:I)=R:(R+t^{a_4-a_1}\ol{R})=I_1$.

(d-3): Suppose that $a_1 + a_3 = 2a_2$ and $\alpha = -\beta^2$. Set $s=a_2-a_1$. Note that we have \eqref{Hform}. Hence, 
\begin{align*}
(t^{a_1}, t^{a_2}) + \fkc=&(t^{a_1}, t^{a_1+s}) + \fkc = (t^{a_1} - \beta t^{a_1+s}, t^{a_1+s}) + \fkc \\
\cong &  (1+\beta t^s)[(t^{a_1} - \beta t^{a_1+s}, t^{a_1+s}) + \fkc]\\
=& (t^{a_1} - \beta^2 t^{a_1+2s}, t^{a_1+s}+\beta t^{a_1+2s}) + (1+\beta t^s)\fkc \\
=& (t^{a_1} +\alpha t^{a_1+2s}, t^{a_1+s}+\beta t^{a_1+2s})  + \fkc \\
=& I.
\end{align*}
\end{proof}

By Claims \ref{claim2} and \ref{claim3}, in the cases (d) and (e), a reflexive ideal $I$ is isomorphic to some monomial ideal containing $\fkc$, respectively. Therefore, for each cases (a)-(e), $I$ is isomorphic to some monomial ideal containing $\fkc$.
\end{proof}

\begin{ex}\label{ex76}
	Let $e\ge 5$ be an integer and set $R=k[[t^e, t^{e+1}, t^{e+2}, t^i \mid 2e+5 \le i \le 3e-1]]$, a numerical semigroup ring over an infinite field $k$. Then $\T(R)$ is infinite, but $\Ref_1(R)$ is finite.
\end{ex}

\begin{proof}
	This is the case where $n=4$, $a_1=e$, $a_2=e+1$, $a_3=e+2$, and $a_4=2e$. It follows that $2a_3-a_1=e+4<2e=a_4$ and $a_2-a_1+a_3=e+3<2e=a_4$.
	Hence, the conclusion follows from Corollary \ref{cor32} and Theorem \ref{thm4.1}.
\end{proof}

We note one of the easiest examples arising from Example \ref{ex76}.

\begin{ex}
Let $R=k[[t^5, t^6, t^7]]$ be a numerical semigroup ring over an infinite field $k$. Then $\T(R)$ is infinite, but $\Ref_1(R)$ is finite.
\end{ex}


\begin{ex}\label{xxx}
	Let $e\ge 4$ be an integer and set $R=k[[t^e, t^{e+1}, t^{2e-2}, t^i \mid 2e+3 \le i \le 3e-3]]$, a numerical semigroup ring over an infinite field $k$. Then $\Ref_1(R)$ (and hence $\T(R)$) is infinite.
\end{ex}

\begin{proof}
	This is the case where $n=4$, $a_1=e$, $a_2=e+1$, $a_3=2e-2$, and $a_4=2e$. It follows that $2a_3-a_1=3e-4\ge 2e=a_4$ since $e\ge 4$. We also have $a_2-a_1+a_3=2e-1<2e=a_4$. 
	Hence, the conclusion follows from Corollary \ref{cor32} and Theorem \ref{thm4.1}.
\end{proof}



\end{document}